\documentclass{amsart}
\usepackage{tikz}
\usepackage{amsmath}
\usepackage{amssymb}
\usepackage{amsthm}

\DeclareMathOperator{\Aut}{Aut}
\DeclareMathOperator{\Sym}{Sym}

\DeclareMathOperator{\id}{id}
\DeclareMathOperator{\Triv}{Triv}

\DeclareMathOperator{\supp}{supp}

\DeclareMathOperator{\Hdim}{dim_{\mathcal{H}}}

\DeclareMathOperator{\wreath}{wr}
\DeclareMathOperator{\hite}{ht}

\newtheorem{theorem}{Theorem}

\newtheorem{cor}[theorem]{Corollary}
\newtheorem{proposition}[theorem]{Proposition}
\newtheorem{lemma}[theorem]{Lemma}
\theoremstyle{definition}
\newtheorem{example}[theorem]{Example}
\newtheorem*{definition}{Definition}
\newtheorem{remark}[theorem]{Remark}
\newtheorem{question}[theorem]{Question}

\title{Nearly Maximal Hausdorff Dimension in Finitely Constrained Groups}

\author[A. Penland]{Andrew Penland}
\address{Department of Mathematics and Computer Science \\
Western Carolina University \\
Cullowhee, NC 28723 USA}
\email{adpenland@email.wcu.edu}

\date{\today}

\begin{document}

\begin{abstract}
This work continues the study of the properties of finitely constrained groups of binary tree automorphisms in terms of their Hausdorff dimension. We prove that there are exactly $2^{2d-3}$ finitely constrained groups of binary tree automorphisms with pattern size $d$ and having Hausdorff dimension $1 - \frac{2}{2^{d-1}}$. As part of this proof, we describe the finite patterns that can define such groups, which leads to the fact that all finitely constrained groups of nearly maximal Hausdorff dimension have additive portraits.  Additionally, we give an upper bound, in terms of the pattern size $d$, on the number of topologically finitely generated instances with nearly maximal Hausdorff dimension for a given $d$, by applying corollaries of the criteria of Bondarenko and Samoilovych. We also construct a new family of examples of finitely constrained, topologically finitely generated groups with nearly maximal Hausdorff dimension. We conclude by positing several open questions.
\end{abstract}


\maketitle

\section{Introduction}

In this paper, we consider finitely constrained groups defined by patterns of size $d$ and having Hausdorff dimension $1 - \frac{2}{2^{d-1}}$. This particular value is interesting because it is the largest possible for a topologically finitely generated, finitely constrained group of binary tree automorphisms with pattern size $d$.This work may be viewed as a continuation of a previous joint work with Zoran \v{S}uni{\'c}~\cite{Penland-Finitely-2016}, which considered these properties for finitely constrained groups with pattern size $d$ and Hausdorff dimension $1 - \frac{1}{2^{d-1}}$. We call the value $1 - \frac{2}{2^{d-1}}$ \textit{nearly maximal} Hausdorff dimension for a finitely constrained group defined by patterns of size $d$, since $1 - \frac{1}{2^{d-1}}$ is the largest possible value for such a group. 

In particular, we will prove the following results.

\begin{theorem}\label{t:main-theorem-intro-1}
If $d \geq 2$, there are exactly $2^{2d-3}$ finitely constrained groups of binary tree automorphisms defined by patterns of size $d$ and having Hausdorff dimension $1 - \frac{2}{2^{d-1}}$. 
\end{theorem}

In the proof of Theorem~\ref{t:main-theorem-intro-1}, we provide a description of all essential pattern groups that define finitely constrained groups with nearly maximal Hausdorff dimension, which yields the following observation. 

\begin{theorem}\label{t:main-theorem-intro-2}
If $H$ be a finitely constrained group defined by patterns of size $d$ that has Hausdorff dimension equal to $1 - \frac{2}{2^{d-1}}$, then $H$ has additive portraits. 
\end{theorem}

Finitely constrained groups, also called \textit{groups of finite type}, are a particular class of \textit{self-similar groups of infinite tree automorphisms}. In general, such self-similar groups have emerged as an important source of examples in group theory, including the \textit{first Grigorchuk group}(introduced in ~\cite{Grigorchuk-Burnside-1980}; see~\cite[Chapter 8]{Harpe-Topics-2000} for an overview of its properties), and the \textit{Gupta-Sidki group}(introduced in~\cite{Gupta-Burnside-1983}). These examples have been generalized and extended in many ways, including the classes of \textit{spinal groups} (see~\cite{Bartholdi-Word-2001}), \textit{branch groups} (see~\cite{Bartholdi-Branch-2003}), and \textit{GGS groups} (named in~\cite[Chapter 2]{Baumslag-Topics-1993}). 

Finitely constrained groups can be thought of using the bijective \textit{portrait map} between the group of all infinite binary tree automorphisms and the full shift on an infinite regular binary tree; see Section~\ref{s:background}. With this setup, a finitely constrained group is a group of tree automorphisms whose portraits correspond to a \textit{tree shift of finite type}. Tree shifts of finite type and their generalizations have been studied in computer science (\cite{Aubrun-Tree-2012},~\cite{Aubrun-Sofic-2013},~\cite{Aubrun-Tree-2014}) and symbolic dynamics(\cite{Ban-Tree-2017}, ~\cite{Ban-Entropy-2017}, ~\cite{Ceccherini-Cellular-2013}). 

For each $k \geq 1$, a group $H$ of automorphisms of the infinite binary tree has a natural projection onto $H(k)$, a finite group of automorphisms of a finite tree with $k$ levels. These projections lead to the \textit{profinite metric}, which we use to endow the group of binary tree automorphisms with a topological structure. In~\cite[Section 7]{Grigorchuk-Solved-2005}, Grigorchuk drew attention to three properties regarding a topologically closed group of tree automorphisms, using the first Grigorchuk group $\mathcal{G}$ as a motivating example. Particularly, for a self-similar topologically closed group of tree automorphisms, we can ask: 
\begin{itemize}
\item is the group finitely constrained\footnote{Grigorchuk used the term \textit{group of finite type}}? i.e. can the group can be defined via the condition that finite quotients of a certain size all belong to a certain group of finite tree automorphisms, the \textit{defining patterns of the group}. Grigorchuk showed that the topological closure of $\mathcal{G}$ is a finitely constrained group, defined using patterns of size 4. 
\item is the group \textit{topologically finitely generated}, i.e. does the group contain a finitely generated, topologically dense subgroup? Obviously, the topological closure of any finitely generated group meets this criterion, but groups defined in other ways may not. 
\item what is the the \textit{Hausdorff dimension} of the group? Here we consider the group as a metric space with the profinite metric induced by the filtration of the group by level stabilizers, and we write $\Hdim(K)$ for the Hausdorff dimension of $K$. Grigorchuk calculated the Hausdorff dimension of the closure of the Grigorchuk group as $\frac{5}{8}$~\cite{Grigorchuk-Just-2000}. The general study of Hausdorff dimension in profinite groups was initiated by Abercrombie~\cite{Abercrombie-Subgroups-1994}. Barnea and Shalev~\cite{Barnea-Hausdorff-1997} further considered Hausdorff dimension in profinite groups, and it is a consequence of their Theorem 2.4, that for a topologically closed group $K$ of binary tree automorphisms, 
\begin{equation}\label{e:haus-dim}
\Hdim(K) = \liminf_{n \rightarrow \infty} \frac{\log_2|K(n)|}{\log_2|G(n)|},
\end{equation} where $G$ represents the group of all automorphisms of the infinite rooted binary tree, and $K(n)$ and $G(n)$ represent the finite quotients by level stabilizers(see Section~\ref{s:background} for more information). Ab{\'e}rt and Vir{\'a}g considered Hausdorff dimension specifically in groups of tree automorphisms and showed that for any $r \in [0,1]$, there exists a topologically finitely generated subgroup of $\Aut(X^*)$ with Hausdorff dimension equal to $r$~\cite{Abert-Dimension-2005}. However, their existence results were nonconstructive, relying on probabilistic methods, and the examples are not guaranteed to be self-similar. Grigorchuk's calculation gave a concrete example of a topologically finitely generated, self-similar group with known Hausdorff dimension.
\end{itemize} 

There are many connections between the properties listed above. For instance, most known concrete examples of Hausdorff dimension in self-similar groups are given by finitely constrained groups. As observed in~\cite[Lemma 8]{Penland-Finitely-2016}, the Hausdorff dimension of a finitely constrained group is straightforward to calculate once its defining patterns are known (see Lemma~\ref{l:FCHD-formula} in this work), which suggests pattern size and Hausdorff dimension as convenient parameters for exploration of finitely constrained groups. 

As a consequence of~\cite[Proposition 2.7]{Bartholdi-Rings-2006}, or, independently, of ~\cite[Proposition 6]{Sunik-Hausdorff-2007}, it is known that the Hausdorff dimension of a finitely constrained group of binary tree automorphisms defined by patterns of size $d$ must have the form $\frac{a}{2^{d-1}}$, where $0 \leq a \leq 2^{d-1}$. \v{S}uni{\'c} constructed, for each $d \geq 4$, a topologically finitely generated, finitely constrained group of binary tree automorphisms with pattern size $d$ and Hausdorff dimension $1 - \frac{3}{2^{d-1}}$, and also constructed similar examples for the odd prime case~\cite{Sunik-Hausdorff-2007}. Siegenthaler used these groups in his construction of the first concrete example of a topologically finitely generated group of binary tree automorphisms with Hausdorff dimension equal to 1~\cite{Siegenthaler-Hausdorff-2008}. 

The closures of iterated monodromy groups of post-critically finite quadratic polynomials studied by Bartholdi and Nekrashevych~\cite{Bartholdi-Iterated-2008} include examples of topologically finitely generated, finitely constrained groups defined by patterns of size $d$ and having Hausdorff dimension $1 - \frac{2}{2^{d-1}}$ for each $d \geq 5$. The Hausdorff dimension of these groups  was calculated by Pink~\cite{Pink-Profinite-2013}. The fact that these groups are finitely constrained, and the defining patterns some of the groups, is discussed in~\cite[Section 5]{Penland-Finitely-2016}. Recently, Samoilovych~\cite{Samoilovych-Profinite-2017} has independently given a description of the defining patterns for the topological closures of self-similar iterated monodromy groups of post-critically finite polynomials. 
 
While not all finitely constrained groups are topologically finitely generated, \v{S}uni\'{c} proved that all finitely constrained groups are topologically \textit{countably} generated. He also showed that no infinite, finitely constrained group of binary tree automorphisms defined by patterns of size two is topologically finitely generated~\cite{Sunik-Pattern-2011}. Bondarenko and Samoilovych~\cite{Bondarenko-Finite-2013} provided criteria to determine whether or not a finitely constrained group is topologically finitely generated. They used these criteria in an exhaustive search with the the computer program \texttt{GAP}~\cite{GAP4}, completing the classification of finitely constrained groups of binary tree automorphisms defined by patterns of size four or less. This search yielded 32 topologically finitely generated examples, each of which has pattern size 4 and Hausdorff dimension $5/8$. 

It is known that a finitely constrained group can only have Hausdorff dimension equal to 1 if the group is all of $\Aut(X^*)$, so the value $1 - \frac{1}{2^{d-1}}$ is the largest possible Hausdorff dimension for a finitely constrained subgroup of $\Aut(X^*)$ defined by patterns of size $d \geq 2$. 

\begin{definition} 
Let $H$ be a finitely constrained subgroup of $\Aut(X^*)$ such that $H$ is defined by patterns of size $d$. If $H$ has Hausdorff dimension $1 - \frac{1}{2^{d-1}}$, we say that $H$ has \textit{maximal Hausdorff dimension}.
\end{definition}

In a previous joint work~\cite{Penland-Finitely-2016}, the current author and Zoran \v{S}uni{\'c} studied the finitely constrained groups of binary tree automorphisms with maximal Hausdorff dimension, proving that a group with these properties can not be topologically finitely generated. An explicit description for the essential pattern groups of size $d$ that define such groups was also given (see Theorem~\ref{t:maximal-subgroup-structure} in this work). 

This naturally leads to the next obvious case.

\begin{definition}
If $H$ is a finitely constrained group of binary tree automorphisms with Hausdorff dimension $1 - \frac{2}{2^{d-1}}$, we say that $H$ has \textit{nearly maximal Hausdorff dimension}. 
\end{definition}

For finitely constrained groups with nearly maximal Hausdorff dimension, already the question of topological finite generation is more subtle, with known examples and counterexamples. As noted in \cite{Penland-Finitely-2016}, the constructions of Bartholdi and Nekrashevych in \cite{Bartholdi-Iterated-2008} give, for each $d \geq 5$, exactly $d-2$ distinct examples of topologically finitely generated, finitely constrained groups defined by patterns of size $d$ and having nearly maximal Hausdorff dimension. 

A special class of finitely constrained groups consists of the groups with \textit{additive portraits}, whose portraits form a subgroup of the full tree shift, considered as an infinite direct product of finite abelian groups.  It is known from~\cite{Arzhantseva-Construction-2007} that the portraits of closure of the first Grigorchuk group are not additive. In~\cite{Siegenthaler-Equations-2012}, Siegenthaler and Zugadi-Reizabal showed that for an odd prime $p$, all non-symmetric GGS groups of $p$-adic automorphisms have additive portraits. It follows from~\cite[Theorem 4.1]{Penland-Finitely-2016} that all finitely constrained groups of binary tree automorphisms with maximal Hausdorff dimension have additive portraits. 

\textbf{Acknowledgments.} Much of this work is adapted from the author's doctoral dissertation~\cite{Penland-Dissertation-2015}, supervised by Zoran \v{S}uni{\'c}. The author wishes to thank Professor \v{S}uni{\'c} for his patient guidance and encouragement, as well as his extensive mathematical expertise. The author also wishes to thank Rostislav Grigorchuk, Volodymyr Nekrashevych, and Renata Ivanek,  who provided many valuable comments while serving on the author's doctoral committee. Finally, the author also wishes to thank David Carroll and Tullio Ceccherini-Silberstein for helpful conversations related to self-similar groups and symbolic dynamics.

\section{Background}\label{s:background}

This section establishes background and notation. There are numerous subsections,  allowing the reader to skim or skip familiar topics. 

\subsection{Trees and Symbolic Dynamics}\label{ss:trees-dynamics}

Here we give essential notions for symbolic dynamics on regular infinite trees, viewed as the Cayley graphs of finitely generated free semigroups. This is a special case of symbolic dynamics on arbitrary semigroups as discussed in~\cite{Coornaert-Symbolic-2006}. 

If $X$ is a finite set and $n$ is a natural number, a \textit{word of length n} in $X$ is a function from  $\{0, 1, \ldots, n-1\}$ to $X$. We write a word of length $n$ as finite string $x_0x_1x_2\ldots x_{n-1}$, with $x^k$ representing a word consisting of $x \in X$ repeated $k$ times. We write $\epsilon$ for the empty word. By convention, we take $0^0 = 1^0 = \epsilon$. 

We write $X^n$ for the set of all words of length exactly equal to $n$ in $X$, $X^{(n)}$ for the set of all words in $X$ of length less than $n$, $X^{[n]}$ for the set of all words in $X$ of length less than or equal to $n$. If $K \subseteq \{0,1,2\ldots,k\}$, we let $X^K = \bigcup_{k \in K} X^k$. Finally, we use $X^*$ to denote the infinite set consisting of all finite words in $X$. 

The set $X^*$ is a monoid with the binary operation given by concatenation. The \textit{(right) Cayley graph} of this  monoid has the elements of $X^*$ as its vertices, with directed edges of the form $(w,wx)$ for $w \in X^*$ and $x \in X$. This graph is an infinite tree with the empty word $\epsilon$ as the root. It is also useful to view the set $X^{[n]}$ as a graph, with a directed edge from $v$ to $vx$ for all $v \in X^{(n)}$. 

If $A$ is a non-empty finite alphabet, the \textit{full shift of A over $X^*$}, denoted $A^{X^*}$ , consists of all functions from $X^*$ to $A$. Considering $X^*$ as an infinite tree, $A^{X^*}$ is also called a \textit{full tree shift}. For $f \in A^{X^*}$, we use $f_{(w)}$ to represent the value of $f$ at $w$. An element of $A^{X^*}$ is called a \textit{configuration}. A \textit{pattern of size d} is a function from the finite set $X^{(d)}$ to $A$, for some positive integer $d$.

The semigroup $X^*$ has a \textit{(right) shift action} on $A^{X^*}$ given by
\[
[\sigma_w(f)]_{(v)} = f_{(wv)}
\] for $f \in A^{X^*}$ and $v,w \in X^*$. A subset $T \subseteq A^{X^*}$ is called \textit{self-similar} (or \text{shift-invariant}) if $\sigma_w(t) \in T$ for all $w \in X^*$ and $t \in T$. 

In the case when the alphabet $A$ is a group, the full shift $A^{X^*}$ naturally inherits the structure of a group, with the group operation given pointwise by 
\begin{equation}\label{l:abelian-group-operation}
(ab)_{(w)} = a_{(w)}b_{(w)}
\end{equation}
for all $w \in X^*$. The set $A^{X^{(d)}}$ inherits an analagous group structure, with the group operation given by Equation~\ref{l:abelian-group-operation}, restricted to $w \in X^{(d)}$.

In this work, we will only consider the case when $A$ is $C_2$, the cyclic group of order two, with elements $\id$ and $\sigma$, written additively, and $X = \{0,1\}$ unless otherwise noted. In this case, we call $A^{X^*}$ the \textit{full binary tree shift group}, and we write $\oplus$ for the pointwise operation given in Equation~\ref{l:abelian-group-operation}. For $n \geq 1$, there is an obvious \textit{projection homomorphism} $q_{n}: A^{X^*} \rightarrow A^{X^{(n)}}$ given by $[q_{n}(f)]_{(w)} = f_{(w)}$ for all $w \in X^{(n)}$. Hence, we view  $A^{X^*}$ as a profinite abelian group, as the projective limit of the system induced by the projection maps $q_n$. The corresponding \textit{profinite metric} $d_F$ on this group is given by
\[
d_{F}(f_1,f_2) = \begin{cases}
0, \text{ if } f_1 = f_2 \\
\frac{1}{|A^{X^{(n)}}|}, \text{where $n$ is the least positive integer such that $q_n(f_1) \neq q_n(f_2)$}. 
\end{cases}
\]

A \textit{binary tree subshift} is a subset of the full binary tree shift group which is both shift-invariant and topologically closed in the profinite metric. A \textit{pattern $p$} of size $d$ \textit{appears in} a configuration $f$ if there exists $w \in X^*$ such that $q_d(f_w) = y$. If $F$ is a set of patterns, we can define $\mathcal{T}_F$, the \textit{tree subshift defined by F} as, the set of all configurations in which no patterns from $F$ appears. We call $F$ the \textit{defining set of forbidden patterns} of the shift $\mathcal{T}_F$. 

Every tree shift has a defining set of forbidden patterns. If a tree shift $\mathcal{T}$ can be defined by some \textit{finite} set of forbidden or allowed patterns, then $\mathcal{T}$ is called a \textit{tree shift of finite type}. In this case, these patterns can be taken to be all of the same size.

\subsection{Group Theory}\label{ss:group-theory}

We assume that the reader is familiar with some basic notions of group theory, including commutators, group homomorphisms, group actions, $p$-groups, etc. at the level of the first two chapters of~\cite{Hungerford-Algebra-1980}. In this subsection, will also provide background about the more specialized topics of $p$-groups and self-similar groups. For more on $p$-groups, see the book by Leedham-Green and McKay~\cite{Leedham-Structure-2002}. An overview of self-similar groups is available in the monograph by Nekrashevych~\cite{Nekrashevych-Self-2005}.

In this work, all finite groups considered will be $2$-groups. If $G$ is a group and $S \subseteq G$ such that $S$ generates $G$, we write $G = \langle S \rangle$.
If $H$ is a group and $h,k \in H$, the \textit{conjugate of h by k} is the element $k^{-1}hk$, denoted $h^k$. If $K$ is a subgroup of $H$, the \textit{normal closure of $K$ in $H$} is the smallest normal subgroup of $H$ that contains $K$. It is well-known that if $K = \langle T \rangle$ for some set $T$, then the normal closure of $K$ in $H$ is generated by the set $\{t^h  \; \mid \; t \in T, h \in H  \}$. If $h$ and $k$ are elements of $H$, we write $[h,k]$ for the commutator $h^{-1}k^{-1}hk$. The \textit{commutator subgroup of H}, denoted $H'$ or $[H,H]$ is the group generated by all commutators in $H$.  If $H$ is a $p$-group, then the \textit{Frattini subgroup} $\Phi(H)$ is the group generated by commutators and $p$'th powers in $H$, and $\Phi(H)$ is the smallest normal subgroup such that the quotient $H/\Phi(H)$ is an elementary abelian $p$-group.

We use function notation for left group actions -- if $H$ is a group with a left action on a set $S$, we write $h(s)$ for the action of $h \in H$ on $s \in S$, whereas we would write $s^h$ for a right action. Although we use $\id$ and $\sigma$ for the elements of the abelian group $C_2$, we use $\id$ for the identity of any group we encounter, relying on context for clarity. All addition will be modulo two. 

The group $\Aut(X^*)$ consists of all graph automorphisms of the infinite binary rooted tree $X^*$. For $d \geq 1$, the group $\Aut(X^{[d]})$ consists of all automorphisms of the finite tree $X^{[d]}$. Henceforth, we use $G$ for $\Aut(X^*)$ and $G(d)$ for $\Aut(X^{d})$. 

\begin{remark}\label{r:generating-sets-Gd}
The groups $G(d)$ can be realized as a Sylow $2$-subgroups of $\Sym(2^d)$, and are instances of the groups considered by Kaloujnine in~\cite{Kaloujnine-Structure-1948}. A natural generating set  for $G(d)$ is the set $\{a_0, a_1, \ldots, a_{d-1} \}$, where $a_i$ swaps each word of the form $0^{i}0w$ with $0^{i}1w$, leaving all other words unchanged. Each $a_i$ has order two. 
\end{remark}

If $H$ is a group of either infinite or finite binary tree automorphisms, the \textit{level k stabilizer of H} is the subgroup which fixes all words of length $k$,. We use $H_k$ to denote the level $k$ stabilizer of $H$. For each $k \geq 1$, there is a natural homomorphism $\pi_{k}: G \rightarrow G(k)$ given by restriction of $g$ to the words in $X^{[k]}$. The kernel of $\pi_k$ is exactly $G_k$. There is an analogous homomorphism $\pi_{n,k}: G(n) \rightarrow G(k)$ for any $n \geq k$ where $\pi_{n,k}(g)$ given by the restriction of $g \in G(n)$ to $X^{[k]}$. If $H$ is a subgroup of $G$ (resp. $G(n)$), we write $H(k)$ for the image of $H$ under $\pi_k$ (resp. $\pi_{n,k})$. 

The homomorphisms $\pi_{n,k}$ determine $G$ as a profinite topological group, with the \textit{profinite metric on $G$} given by $d(g,h) = 0$ for $g = h$, and for $g \neq h$ by
\[
d_{G}(g,h) = \frac{1}{|G(n)|}, \text{where $n$ is the least positive integer such that
$\pi_n(g) \neq \pi_n(h)$}.
\]
 
For $g \in G$ and $w \in X^*$, the tree automorphism $g$ naturally yields an automorphism of the tree $g(w)X^*$. We define the \textit{section of g at w} to be the unique binary tree automorphism $g_w$ whose action on an element $v \in X^*$ is given by $g(wv) = g(w)g_{w}(v)$. 

\begin{definition}
If $S$ is a subset of $G$, we say that $S$ is \textit{self-similar} if $h_w \in H$ for all $h \in S$ and $w \in X^*$. If $S$ is self-similar and a subgroup of $G$, we say that $S$ is a \textit{self-similar group}.
\end{definition}

Similarly to sections in $G$, if $g \in G(d)$ and $w \in X^{(d)}$, we define the \textit{finite section of g at w} as the unique element of $G(d - |w|)$ whose action on $v \in X^{[d-|w|]}$ is defined by $g(wv) = g(w)g_w(v)$. A subgroup $P$ of $G(d)$ is an \textit{essential pattern group} if for each $p \in P$ and $i \in \{0,1\}$, there exists $q_i$ such that $\pi_{d-1}(q_i) = p_i$.  In other words, an essential pattern group is the finite analog of a self-similar group, containing all finite sections of all elements. 

\subsection{Patterns and Portraits}\label{ss:patterns}

We write $\alpha$ for the homomorphism $\pi_{1}$, i.e. the map $\alpha: G \rightarrow C_2$ given by restricting the action of $g$ to words of length one, so
\[
\alpha(g) = \begin{cases} \sigma, \text{ if } g(0) = 1 \\ \id, \text{ otherwise }. \end{cases}
\] If $g \in G$ and $w \in X^*$, or if $g \in G(d)$ and $|w| < d$, then $\alpha_{(w)}(g)$ is defined to be $\alpha_{(w)}(g) = \alpha(g_w)$. 

If $S$ is a finite subset of $X^*$ and $g \in G$ or $g \in G(d)$, we let 
\[
\alpha_{S} = \sum_{w \in S} \alpha_{(w)}(g)
\]

so that $\alpha_S$ determines the number, modulo two, of elements in $S$ with $\alpha_{s}(g) \neq \id$. 

If $g \in G$ or $g \in G(d)$, we write $\alpha_k$ for $\alpha_{X^k}(g)$. If $g \in G(d)$ and $J \subseteq \{0, 1, \ldots, n \}$ for some nonnegative integer $n$, we let $\alpha_{J}(g) = \displaystyle\sum_{k \in J} \alpha_k(g)$. 

If $g \in G$ (resp. $G(d)$), we define the \textit{support of g} to be the set of $w \in X^*$ (resp. $X^{(d)})$ such that $\alpha_{(w)}(g) \neq \id$. We write $\supp(g)$ for the support of $g$, which uniquely defines an element in the groups we consider. 

\begin{remark}\label{r:abelianization}
It is well-known that the largest abelian quotient of $G(d)$ is $(C_2)^{d}$, via the homomorphism $\lambda: g \mapsto [\alpha_{i}(g)]_{i=0}^{d-1}$.
\end{remark}

The $\alpha_{J}$ homomorphisms are crucial to characterizing the essential pattern groups used to define finitely constrained groups of maximal Hausdorff dimension.

\begin{theorem}[Theorem 4.1,~\cite{Penland-Finitely-2016}]\label{t:maximal-subgroup-structure}
Let $G_P$ be a finitely constrained group defined by an essential pattern group $P$ of patterns size $d$, $d \geq 2$. The following are equivalent.
\begin{itemize}
\item There exists $J \subseteq \{0,1,\ldots, d-1\}$ with $d-1 \in J$, so that $P = \ker \alpha_{J}$
\item $G_P$ has maximal Hausdorff dimension
\item $P$ is a maximal subgroup of $G(d)$ which does not contain $[G(d),G(d)]$
\item $P$ is a maximal subgroup of $G(d)$ which does not contain the level $d-1$ stabilizer $G(d)_{d-1}$.
\end{itemize}
\end{theorem}

\begin{remark}\label{r:aj-or-ajad-1}
Suppose that $P$ is a maximal subgroup of $G(d)$, so that $P = \ker \alpha_J$ for some $J \subseteq \{0,1,\ldots, d-1\}$. Suppose $P$ is also an essential pattern group. It is immediate from Theorem~\ref{t:maximal-subgroup-structure} that for $j \in \{0,1,\ldots, d-2\}$, the generator $a_j \in P$ if and only if $j \not\in J$, and $a_ja_{d-1} \in P$ if and only if $j \in J$. 
\end{remark}

The homomorphism $\alpha$ is also used to define the \textit{portrait map}, which provides the correspondence between group automorphisms and configurations in the full tree shift. 

\begin{definition}\label{d:portrait-map}
The \textit{portrait map} $\rho: \Aut(X^*) \rightarrow (C_2)^{X^*}$ is given by $\rho(g)_{(w)} = \alpha_{(w)}(g)$.
\end{definition}

\begin{remark}
In general, $\rho(gh) \neq \rho(g) \oplus \rho(h)$, i.e. the portrait map is \textit{not} a homomorphism from $G$ to $(C_2)^{X^*}$. In fact, if $H$ is a subgroup of $G$, $\rho(H)$ need not even be a subgroup of $(C_2)^{X^*}$. 
\end{remark}

The preceding observation motivates the definition of \textit{additive portraits}.

\begin{definition}
If $H$ is a subgroup of $\Aut(X^*)$, we say that $H$ \textit{ has additive portraits} if $\rho(H)$ is a subgroup of $(C_2)^{X^*}$.
\end{definition}

In other words, $H$ has additive portraits if for any $g, h \in H$, the point $\rho(g) \oplus \rho(h)$ is an element of $\rho(H)$. Note that although the portrait map $\rho: \Aut(X^*) \rightarrow (C_2)^{X^*}$ is not an isomorphism in general, it \textit{is} an isometry between the metric spaces $(G, d_{G}$) and $((C_2)^{X^*}$, $d_{F}$).

\begin{definition}
Let $P$ be an essential pattern group with pattern size $d$. The \textit{group defined by $P$} is the subgroup of $G$ given by
\[
G_P = \{ g \in G \mid \pi_{d}(g_w) \in P \text{ for all } w \in X^* \}.
\] 
\end{definition}

\begin{definition}
Let $H$ be a subgroup of $\Aut(X^*)$. We say that $H$ is a \textit{finitely constrained group} if there exists $k \geq 1$ and an an essential pattern group $P \leq G(k)$ such that $H = G_P$. If $d$ is the \textit{minimal} natural number such that there exists an essential pattern group $P$ of size $d$ with $H = G_P$, then we say that $H$ is \textit{defined by patterns of size d.}
\end{definition}

\begin{remark}
It is not hard to see that the group $G_P$ is topologically closed and self-similar. If $H$ is a subgroup of $\Aut(X^*)$, then $H$ is a closed, self-similar subgroup of $\Aut(X^*)$ if and only if $\rho(H)$ is a tree shift, and $H$ is a finitely constrained group if and only if $\rho(H)$ is a tree shift of finite type. 
\end{remark}

For $g \in G$ and $w \in X^*$, we define the \textit{branch of g at w}, denoted $\delta_{w}(g) \in G$ to be unique element given by

\[
[\delta_{w}(g)]_{z} = \begin{cases}
g_v, \text{ if } z=wv \\
\id, \text{ otherwise. }
\end{cases}
\]

%Don't think we need this anyhwere: It is clear that if $g \in G_n$ and $w \in X^k$, $\delta_{w}(g) \in G_{n+k}$. 

\begin{definition}
If $H$ is an infinite subgroup of $\Aut(X^*)$ And $K$ is a subgroup of $H$, we say that $H$ is \text{regular branch over K} if $K$ is a normal subgroup of $H$, $K$ has finite index in $H$ and $\delta_{x}(k) \in K$ for all $k \in K$ and $x \in X$. 
\end{definition}

The following theorem characterizes finitely constrained groups of binary tree automorphisms as regular branch groups over level $d$ stabilizers.   

\begin{theorem}\label{t:branch-equals-finitely-constrained}
For an infinite group $H$ of binary tree automorphisms, the following are equivalent. 
\begin{itemize}
\item[i.] $H$ is a finitely constrained group, defined by patterns of size $d$
\item[ii.] $H$ is the closure of a regular branch group $K$, branching over its level $d-1$ stabilizer, $K_{d-1}$. 
\end{itemize}
\end{theorem}

The proof of (i.) $\implies$ (ii.) in Theorem~\ref{t:branch-equals-finitely-constrained} is given in~\cite[Proposition 7.5]{Grigorchuk-Solved-2005}, while the proof of (ii.) $\implies$ (i.) is given in ~\cite[Theorem 3]{Sunik-Hausdorff-2007}.

\subsection{Hausdorff dimension in finitely constrained groups}

If $G_P$ is a finitely constrained constrained group defined by patterns of size $d$, then Bondarenko and Samoilovych observed in the proof of ~\cite[Proposition 1]{Bondarenko-Finite-2013} that for $n \geq d$, we have
\begin{equation}\label{e:finite-quotient-growth}
|G_P(n)| = |P||P_{d-1}||^{|X| + |X|^2 + \ldots + |X|^{n-d}}
\end{equation}

As first noted in ~\cite[Lemma 8]{Penland-Finitely-2016}, applying this count to the formula given in Equation~\ref{e:haus-dim}, it is not hard to see that the Hausdorff dimension of a finitely constrained group is completely determined by the size of $P_{d-1}$. 

\begin{lemma}~\label{l:FCHD-formula}
 If $P$ is an essential pattern subgroup of $G(d)$, then $\Hdim(G_P) = \frac{\log_2 |P_{d-1}|}{2^{d-1}}$.
\end{lemma}

\begin{proof}
Since $G_P$ is a pro-$2$ group, we know from Equation~\ref{e:haus-dim} that the Hausdorff dimension of $G_P$ is given by
\[
 \Hdim(G_P) = \liminf_{n \rightarrow \infty} \frac{\log_2 |G_P(n)|}{\log_2 |G(n)|}.
 \]
Noting that $\log_2 |G_P(n)| = 2^{n} - 1$ and
applying the result in Equation~\ref{e:finite-quotient-growth}, we then calculate that
\begin{align*}
\Hdim(G_P) &= \liminf_{n \rightarrow \infty} \frac{\log_2 \left( |P||P_{d-1}|^{2 + \ldots + 2^{n-d}} \right)}{2^{n}-1} \\
&= \liminf_{n \rightarrow \infty} \frac{\log_2|P| + (2 + \ldots + 2^{n-d})\log_2| P_{d-1}|}{2^n - 1} 
\end{align*}
Since $2 + 2^2 + \ldots + 2^{n-d} = 2^{n-d+1} - 1$, we then have
\begin{align*}
\Hdim(G_P) &= \liminf_{n \rightarrow \infty} \frac{\log_2 |P|}{2^n - 1} + \frac{2^{n-d+1} \log_2|P_{d-1}|}{2^{n} - 1} - \frac{1}{2^n - 1} \\
&= \liminf_{n \rightarrow \infty} \frac{\log_2 |P_{d-1}|}{2^{d-1}}\\
&= \frac{\log_2 |P_{d-1}|}{2^{d-1}}
\end{align*}
\end{proof}

\begin{remark}
Viewing $X^*$ as a free monoid with two generators, the full shift $A^{X^*}$ is a generalization of the one-sided,one-dimensional full shift $A^{\mathbb{N}}$, since $\mathbb{N}$ is a free monoid with one generator. For one-sided,one-dimensional subshifts, the Hausdorff dimension function agrees up to a multiplicative constant with the well-studied \textit{topological entropy} - a fact first observed by Fursteunburg in~\cite{Furstenburg-Disjointness-1967} and further by Simipson in explored in~\cite{Simpson-Symbolic-2017}. 
\end{remark}

We record now some basic but useful facts about essential pattern groups, finitely constrained groups, and Hausdorff dimension.

\begin{proposition}\label{p:essential-pattern-self-similar}
A group $P \subseteq G(n)$ is an essential pattern group if and only if there exists a self-similar group $H$ such that $H(n) = P$. 
\end{proposition}

\begin{proof}
If $P$ is an essential pattern group, then the group $G_P$ is a self-similar group such that $G_P(n) = P$. If $H$ is a self-similar group and $H(n) = P$, then $p = \sigma^i(p_0,p_1)$ is the image of some $h = \sigma^i(h_0,h_1)$ under $\pi_n$, such that $p_0 = \pi_{n-1}(h_0)$ and $p_1 = \pi_{n-1}(h_1)$. Since $H$ is self-similar, $h_0, h_1 \in H$, so for $i = 0,1$, we have $\pi_{n}(h_i)$ as an element of $P$ satisfying $\pi_{n-1}(\pi_n(h_i)) = p_i$. Thus $P$ is an essential pattern group. 
\end{proof}

\begin{proposition}\label{p:bigger-pattern-sizes-define-the-same-group}
Let $Q$ be an essential pattern group defined by patterns of size $k$. Then the following hold.
\begin{itemize}

\item[(i.)] $G_Q(k) = Q$ 
\item[(ii.)] For all $n \geq k$, $G_{(G_Q(n))} = G_Q$ 
\item[(iii.)] If $P$ is an essential pattern group of size $m \geq k$ such that $P = G_Q(m)$, then $G_P = G_Q$. 
\item[(iv.)] If $P$ is an essential pattern group of size $m \geq k$ such that $P(k) = Q$, then $P \subseteq G_Q(m)$. 
\item[(v.)] If $n \leq k$ and $P \subseteq G(n)$ such that $P = \pi_{n,k}(Q)$, then $P$ is an essential pattern group. 
\end{itemize}
\end{proposition}

\begin{proof}
(Proof of (i.)) If $g \in G_Q$, then by definition $\pi_k(g) \in Q$, so $G_Q(k) \subseteq Q$. On the other hand, if $q \in Q$, then since $Q$ is an essential pattern group, it is possible to build an element $g \in G_Q$ such that $\pi_k(g) = q$. Thus $G_Q(k) = Q$. (Proof of (ii.)) Let $n \geq k$. If $g \in G_{G_Q(n)}$, then for all $w \in X^* \pi_n(g_w) \in G_Q(n)$, so $\pi_k(g_w) \in G_Q(k) = Q$. It follows that $g \in G_Q$, so $G_{G_Q(n)} \subseteq G_Q$. Now suppose $g \in G_Q$. Since $G_Q$ is a self-similar group, for all $w \in X^*$, we have $g_w \in Q_G$, so $\pi_n(g_w) \in G_Q(n)$. It follows that $g_w \in G_{G_Q(n)}$, so $G_Q \subseteq G_{G_Q(n)}$. (Proof of (iii.)) If $P = G_Q(n)$, then it follows from (i,) and (ii.) that $G_P = G_{G_Q(n)} = G_{G_Q(k)} = G_Q$. (Proof of (iv.)) If $P(k) = Q$, then $\pi_{m.k}(p) \in Q$ for all $p \in P$. Thus $P \subseteq G_Q(m)$. 
(Proof of (v.)) In this case, $P = \pi_{k}(G_Q)$, by (i.) and (ii.), and since $G_Q$ is a self-similar group, $P$ is an essential pattern group by Proposition~\ref{p:essential-pattern-self-similar}.
\end{proof}

\begin{proposition}\label{p:metric-facts}
Let $d \geq 2$, let $P$ be an essential pattern subgroup of $G(d)$, and let $G_P$ be the finitely constrained group defined by $P$. Then the following hold. 
\begin{enumerate}
\item[(i.)] For $n \in \mathbb{N}$ and let $g, h\in G$, we have that $d(g,h) < \frac{1}{|G(n)|}$ if and only if $\pi_n(g) = \pi_n(h)$. 
\item[(ii.)] If $H$ is a subgroup of $G$, then $g \in \overline{H}$ if and only if $\pi_n(g) \in H(n)$ for all $n \in \mathbb{N}$. 
\item[(iii.)] If $H$ is a self-similar subgroup of $G$, then $H \leq G_{H(n)}$ for all $n \in \mathbb{N}$.
\item[(iv.)] If $H$ is a self-similar subgroup of $G$ and $m < n$, then $G_{H(m)} \geq G_{H(n)}$. 
\item[(v.)] If $H$ is a self-similar subgroup of $G$, then $\overline{H} = \bigcap_{n \in \mathbb{N}} G_{H(n)}$. 
\end{enumerate}
\end{proposition}

\begin{proof}
\begin{enumerate}
\item[(i.)] This follows immediately from the definition of the profinite metric on $G$. 
\item[(ii.)] Let $g \in G$ and suppose $g \in \overline{H}$. For any $n \in \mathbb{N}$, there exists $h_n$ such that $d(g,h_n) < \frac{1}{|G(n)|}$, and thus $\pi_{n}(g) = \pi_{n}(h_n) \in H(n)$. Thus $\pi_n(g) \in H(n)$ for all $n \in \mathbb{N}$. Now suppose $\pi_n(g) \in H(n)$ for all $n \in \mathbb{N}$. Then for each $n \in \mathbb{N}$, there exists $h_n$ such that $\pi(g) = \pi(h_n)$, and $g \in \overline{H}$. 
\item[(iii.)] Let $h \in H$ and $n \in \mathbb{N}$. Since $H$ is self-similar, $h_w \in H$ for all $w \in X^*$, so $\pi_n(h_w) \in H(n)$ for all $w \in X^*$. Thus $h \in G_{H(n)}$ by (iii.)
\item[(iv.)] If $g \in G_{H(n)}$, then $\pi_n(g_w) \in H(n)$ for all $w \in X^*$, so $\pi_m(g_w) \in H(m)$ for all $w \in X^*$.
\item[(v.)] First we show that $\overline{H} \subseteq \bigcap_{n \in \mathbb{N}} G_{H(n)}$. Each $G_{H(n)}$ is a closed set which contains $H$ by (iv.), so $\overline{H} \subseteq G_{H(n)}$ for all $n$. If $g \in \bigcap_{n\in\mathbb{N}} G_{H(n)}$, then $\pi_n(g) \in H(n)$ for all $n$, so $g \in \overline{H}$ by (ii.)
\end{enumerate}
\end{proof}

\subsection{Wreath Products, Automata, and Uniseriality}\label{ss:wreath-products-uniseriality} 

Suppose $A$ and $N$ are groups such that $A$ has a right action on $N$ by automorphisms. The \textit{semi-direct product} $A \ltimes N$ is the set of tuples $(a,n)$, $a \in A$, $n \in N$, with binary operation given by
\[
(a_1,n_1)(a_2,n_2) = (a_1a_2, n_1^{a_2}n_2).
\]

For a group $H$ and a set $Y$, $H^{Y}$ denotes the set of all functions from $Y$ to $H$. If $A$ has a left action on a set $Y$, there is a natural right action of $A$ on $H^{Y}$ given by $(f^a)_{(y)} = f_{a(y)}$. Given an action of a group $A$ on a set $Y$, along with a group $H$, the \textit{wreath product} $A \wreath_{Y} H$ is the semi-direct product $A \ltimes H^Y$. 

The group $G$ naturally decomposes as the wreath product $G = C_2 \wreath_{X} G$, so any element $g \in G$ can be written as $\sigma^i(g_0,g_1)$, where $i \in \{0,1\}$ and $g_0, g_1 \in G$. Of course, $\sigma^i \in C_2$ is the value of $\alpha(g)$ from Subsection~\ref{ss:patterns}, and $g_0$ and $g_1$ are the sections of $g$ at $0$ and $1$ respectively, as disussed in Subsection~\ref{ss:group-theory}. We call this way of writing $g$ its \textit{wreath recursion}.

\begin{definition}
A \textit{finite state automaton} is a finite self-similar set. 
\end{definition}

\begin{example}\label{e:grigorchuk-automaton}

The Grigorchuk group is generated by the finite state automaton
\[
a = \sigma(\id,\id), \quad b = (a,c), \quad c = (a,d) \quad d = (\id,b).
\]
whose elements are written using wreath recursion.
\end{example}

\begin{example}[Calculation using wreath recursion]\label{e:wreath-calculation}
Using wreath recursion is very helpful in calculation using tree automorphisms. In general, if $g = (g_0,g_1)$ and $h = \sigma(h_0,h_1)$, we have
\[
gh = (g_0,g_1)\sigma(h_0,h_1) = \sigma(g_0,g_1)^{\sigma}(h_0,h_1) = \sigma(g_1,g_0)(h_0,h_1) = \sigma(g_1h_0,g_0h_1)
\]

and 
\[
hg = \sigma(h_0,h_1)(g_0,g_1) = \sigma(h_0g_0,h_1g_1)
\]

For instance, in the Grigorchuk group, we have
\[
abc = \sigma(\id,\id)(a,c)(a,d) = \sigma(a^2,cd) = \sigma(\id,cd)
\]

as well as

\[
abac = \sigma(\id,\id)(a,c)\sigma(\id,\id)(\id,b) = \sigma(a,c)\sigma(\id,b) = (a,c)^\sigma(\id,b) = (c,a)(\id,b) = (c,ab).
\]

\end{example}

There are many ways to express the finite group $G(d)$ as a wreath product. For our purposes, two such decompositions are especially useful. The first is $G(d) = C_2 \wreath_X G(d-1)$. As is the case with elements of $G$, we can write any $g \in G(d)$ using wreath recursion as $g = \sigma^{i}(g_0,g_1)$, where $i \in \{0,1\}$, and $g_0$ and $g_1$ are elements of $G(d-1)$ (these are the \textit{finite sections} previously discussed).

Another useful way to express $G(d)$ as a wreath product is given by $G(d-1) \wreath_{X^{d-1}} C_2$. Identifying $C_2$ with the finite field with two elements, the vector space $\prod_{w \in X^{d-1}} C_2$ of dimension $2^{d-1}$ over this field, which corresponds to the subgroup $G(d)_{d-1}$. The action of subgroups of $G(d)$ on this vector space has been extensively studied using the notion of {\em uniseriality}. Uniseriality is a general notion for group actions, discussed in detail for $p$-groups in~\cite[Chapter 4]{Leedham-Structure-2002}; see also~\cite{Plesken-Uniserial-1983}.

\begin{definition}
If $P$ is a subgroup of $G(d)$, the {\em $P$-filtration of $G(d)_{d-1}$} is defined to be the sequence of subgroups given by
\begin{align*}
V_P^{(0)} &= G(d)_{d-1}  \\
V_P^{(i+1)} &= [P, V_P^{(i)}]
\end{align*}
\end{definition}

In the context we consider, the action of $P$ on $G(d)_{d-1}$ is said to be {\em uniserial} if $[V_P^{(i)}:V_P^{(i+1)}] = 2$ for all $0 \leq i \leq d-1$. 
Note in particular that uniseriality of the action of $P$ on this vector space that implies $V_P^{(d-1)}$ is trivial. 

We write $V_{d-1}^{(i)}$ for $V_{G(d)}^{(i)}$, and we write $V_{d-1}$ for $V_{d-1}^{(0)}$. 

As a minor technical point, uniseriality for groups of tree automorphisms typically considers the action of subgroups of $G(d-1)$ on $\prod_{w \in X^{d}} C_2$. By adapting our consideration to the action of subgroups of $G(d)$, we are implicitly considering this action of $G(d)$ through a quotient action of $G(d-1)$. We note this for clarity with regard to the literature, but it poses no actual difficulty. 

The following result summarizes some facts about uniseriality found in Proposition 2.1 and Theorem 2.7 of~\cite{Ceccherini-Generalized-2005}.

\begin{theorem}\label{t:practical-uniserial}
\begin{enumerate}
\item The action of $G(d)$ on $V_{d-1}$ is uniserial.
\item If $P$ is any subgroup of $G(d)$ which acts uniserially on $V_{d-1}$, then $V_P^{(i)} = V_{d-1}^{(i)}$ for $0 \leq i \leq 2^{d-1}$. 
\item Suppose $P$ acts uniserially on $W$. If $W$ is any subgroup of $V_{d-1}$ which is normal in $P$, then $W = V_{d-1}^{(i)}$ for some $i$. 
\end{enumerate}
\end{theorem}

\begin{remark}\label{r:uniserial-implies}
Theorem~\ref{t:practical-uniserial} implies that if $P$ acts uniserially on $V_{d-1}$ and the size of the subgroup $P_{d-1}$ is known,there is only one possibility for the subgroup $P_{d-1}$: if $\log_{2}|P_{d-1}| = 2^{d-1}-i$, then it must be the case that $P_{d-1} = V^{(i)}_{d-1}$.
\end{remark}

The patterns of finite tree automorphisms provide a way to determine if the action of the group is uniserial. 

\begin{theorem}[Proposition 4.2.11,~\cite{Leedham-Structure-2002}]\label{t:pattern-uniserial-criterion}
If $P$ is a subgroup of $G(d)$, the action of $P$ on $V_{d-1}$ is uniserial if and only if $\alpha_{k}(P) = C_2$ for all $0 \leq k \leq d-2$.
\end{theorem}

\begin{cor}\label{c:full-pattern-uniserial}
If $P(d-1) = G(d-1)$, then $P$ acts uniserially on $V_{d-1}$. 
\end{cor}
 
Using the uniserial filtration of $V_{d-1}$, the {\em height} of an element $v \in V_{d-1}$ is defined to be $2^{d-1} - k$, where $k$ is the largest integer such that $v \in V_{d-1}^{(k)}$. We write $\hite(v)$ for the height of $v$. From the definition, it is immediate that $\hite(v)$ is invariant under conjugation by elements of $G(d)$. 

Recall from Lemma~\ref{l:FCHD-formula}, if $G_P$ is a finitely constrained group defined by an essential pattern group $P$, the Hausdorff dimension of $G_P$ is determined by immediately by the size of the subgroup $P_{d-1}$, the level $d-1$ stabilizer of $P$. Thus, height and uniserial actions are extremely useful in determining the Hausdorff dimension of finitely constrained groups. 
 
\begin{proposition}\label{p:essential-uniserial-action}
 Let $P$ be a group which acts uniserially on $V_{d-1}$. If $P_{d-1}$ contains an element $v \in V_{d-1}$ with $\hite(v) = 2^{d-1} - k$, then $P$ contains $V_{d-1}^{(k)}$. 
\end{proposition}
 
 \begin{proof}
Since  the normal closure of the group ${\langle v \rangle}$ is a $P$-invariant subgroup of $V_{d-1}$, it follows from Theorem~\ref{t:practical-uniserial} that the normal closure in $P$ of the group $\langle v \rangle$ is equal to $V_{d-1}^{(i)}$, where $i = 2^{d-1} - \hite(v)$. 
 \end{proof}
 
 \begin{cor}\label{c:normal-closure-height}
Let $P$ be an essential pattern group which acts uniserially on $V_{d-1}$. If there exists $p \in P_{d-1}$ with $\hite(p) = k$, then $\dim_{H}(G_P) \geq \displaystyle\frac{k}{2^{d-1}}$. 
 \end{cor}
 
 \begin{proof}
 This follows from Proposition~\ref{p:essential-uniserial-action} and Lemma~\ref{l:FCHD-formula}.
 \end{proof}
 
Using the wreath product decomposition of $G(d)$, if $d \geq 2$, we can decompose any $v \in V_{d-1}$ uniquely as $v = (v_0, v_1)$, where $v_0, v_1 \in V_{d-2}$. It is immediate for $G(1) = C_2$, where $G(1)_{1} = G(1)$, that we have that $\hite(\sigma) = 1$ and $\hite(\id) = 0$. Using these values as a base case, the height of any element $v$ in any group $V_{d-1}$ can then be calculated recursively using the following formula found in~\cite{Ceccherini-Generalized-2005}.

\begin{theorem}[Theorem 2.10, \cite{Ceccherini-Generalized-2005}]\label{t:recursive-ht-formula}

If $v \in V_{d-1}$ can be decomposed as $v = (v_0, v_1)$ where $v_0, v_1 \in G(d-1)_{d-2}$, then

\[
\hite(v) = \begin{cases}

2 \max{ \{\hite(v_0), \hite(v_1) \} }, \text{ if } \hite(v_0) \neq \hite(v_1) \\
h(v_0) + h(v_1) - 1, \text{ if } \hite(v_0) = \hite(v_1)

\end{cases} 
\]

\end{theorem}

For later convenience, we record an observation about two particular elements of large height in $G(d)$ in terms of the standard generators $\{a_0, a_1, \ldots, a_{d-1} \}$ given in Subsection~\ref{ss:group-theory}.

 \begin{cor}\label{c:large-height-elements}
 For any $d \geq 2$, $\hite\left([a_0,a_{d-1}]\right) = 2^{d-1} - 1$ and $\hite\left([a_1,a_{d-1}]\right) = 2^{d-1} - 2$. 
 \end{cor}
 
 \begin{proof}
 For the base case $d=2$, we have 
 \[ \hite([a_0,a_1]) = \hite((\sigma,\sigma)) = \hite(a_0) + \hite(a_0) - 1 = 1 = 2^{2-1} - 1 \]
 
 and 
 \[
 \hite([a_0,a_0]) = \hite((\id,\id)) = 0 = 2^{2-1} - 2.
 \]
 
 Now assume the statement is true for some $d = k$, and consider $d = k+1$. Applying Theorem~\ref{t:recursive-ht-formula}, we have 
 \begin{align*}
 \hite([a_0,a_{k+1-1}] = \hite([a_0,a_{k-1}]) = \hite((a_{k-2},a_{k-2}) = 2\hite(a_{k-2}) - 1 = 2^{(k+1)-1} - 1
 \end{align*}
 where the last equality follows from the induction hypothesis. 
 
 Similarly, we calculate
 \begin{align*}
 \hite([a_1,a_{k+1-1}]) =  ([a_0,a_{k-1}], \id) = 2 \hite([a_0,a_{k-1}]) = 2 (2^{k-1} - 1 ) = 2^{k+1-1} - 2
 \end{align*}
 
 This completes the proof. 
 
 \end{proof}
 
 \begin{proposition}\label{p:must-contain-commutator}
 Let $P$ be an essential pattern group with pattern size $d$ such that. $P(d-1) = G(d-1)$. The group $P$ contains $[a_1, a_{d-1}]$ and does not contain $[a_0,a_{d-1}]$ if and only if $[G(d):P] = 4$.
 \end{proposition}
 
 \begin{proof}
Since $|P| = |P(d-1)||P_{d-1}|$ and $P(d-1) = G(d-1)$, it follows that \[
[G(d):P] = \frac{|G(d-1)||G_{d-1}|}{|P(d-1)||P_{d-1}|}  = \frac{|G(d-1)||G_{d-1}|}{|G(d-1)||P_{d-1}|} = \frac{|G_{d-1}|}{|P_{d-1}|}.
\]
The action of $P$ on $V_{d-1}$ is uniserial by Corollary~\ref{c:full-pattern-uniserial}. If $P$ contains the element $[a_1, a_{d-1}]$, then it follows from Corollary~\ref{c:large-height-elements} and Proposition~\ref{p:essential-uniserial-action} that $\frac{|G_{d-1}|}{|P_{d-1}|} \geq 4$, and $\frac{|G_{d-1}|}{|P_{d-1}|} \geq 2$ if and only if $P$ contains $[a_0,a_{d-1}]$. 
 \end{proof}

\begin{proposition}\label{p:large-stabilizer-subgroups}
For $d \geq 2$, we have

\[
V_{d-1}^{(1)} = \{v \in V_{d-1} \mid \alpha_{d-1}(v) = \id \}
\]

\[
V_{d-1}^{(2)} = \{v \in V_{d-1} \mid \alpha_{0X^{d-2}} = \alpha_{1X^{d-2}}(v) = \id \}
\]
\end{proposition}

\begin{proof}
The elements of $V_{d-1}$ that satisfy $\alpha_{d-1}(v) = \id$ form a normal subgroup of $G(d)$ that has index 2 in $V_{d-1}$, so this set is equal to $V_{d-1}^{(1)}$ by Theorem~\ref{t:practical-uniserial}(iii.). Elements in $V_{d-1}$ that have $\alpha_{0X^{d-2}} = \alpha_{1X^{d-2}}(v) = \id$ also form a normal subgroup of $V_{d-1}$, this normal subgroup has index 2 in $V_{d-1}^{(1)}$ as the kernel of $\alpha_{0X^{d-2}}$, which is a homomorphism when restricted to $V_{d-1}^{(1)}$. Hence, this set must be equal to $V_{d-1}^{(2)}$ by Theorem~\ref{t:practical-uniserial}(iii.).
\end{proof}

\section{Preliminary Results}

\subsection{All Possible Hausdorff Dimensions Occur For Finitely Constrained Groups}

The uniserial filtration discussed in subsection~\ref{ss:wreath-products-uniseriality} allows us to construct finitely constrained groups with any possible Hausdorff dimension. The trivial group has Hausdorff dimension equal to 0, while $\Hdim(G) = 1$, so we do not consider these values.  

\begin{proposition}\label{p:all-haus-dim-occur}
For each $d \geq 2$ and each $a$ such that $0 < a < 2^{d-1}$. There exists a finitely constrained subgroup of $\Aut(X^*)$ with pattern size $d$ and Hausdorff dimension $\frac{a}{2^{d-1}}$.
\end{proposition}

\begin{proof}
Let $d \geq 2$ and $a$ be a positive integer such that $0 < a < 2^{d-1}$. Let $i = 2^{d-1} - a$. We take $H$ to be the subgroup of $G(d)$ generated by the elements $a_0, a_1, \ldots, a_{d-2}$, and we take $N$ to be the group $V_{d-1}^{i}$. Let $P = HN$. By the definition of $H$, we have that $P(d-1) = G(d-1)$, so $P$ is an essential pattern group that acts uniserially on $V_{d-1}$. By the uniseriality of this action, it must be the case that $P_{d-1} = N = V_{d-1}^{(i)}$, so $\log_{2}|P_{d-1}| = a$, whence $\Hdim(G_P) = \frac{a}{2^{d-1}}$ by Lemma~\ref{l:FCHD-formula}.
\end{proof}

\begin{remark}
It is clear from the description of the groups in the previous proof that $P_d^{(i)}$ is a split extension of $G(d)$ by $V_{d-1}^{(i)}$. Using Corollary~\ref{c:easy-cor-2} below, it follows that none of these groups are topologically finitely generated.
\end{remark}

\subsection{Generating Sets for Essential Pattern Groups}

For our purposes, an \textit{extension} of a group $N$ by a group $H$ consists of the following a group $K$, along with an injective homomorphism $\iota: N \rightarrow K$ and a surjective homomorphism $\nu: K \rightarrow H$ such that the image of $\iota$ is equal to the kernel of $\nu$. 

The theory of group extensions, particularly for finite groups, is very well-developed (see, for instance, ~\cite[Chapter 10]{Johnson-Presentations-1997}. We will need very little, but we note that if $P$ is an essential pattern group, then $P$ is an extension of $P_{d-1}$ by $P(d-1)$. 

This observation leads to descriptions of generating sets for $P$. 

\begin{proposition}\label{p:cheap-generating-sets}
Let $P$ be an essential pattern group such that $P(d-1) = G(d-1)$ and $P_{d-1} = V^{(i)}_{d-1}$ for some $0 \leq i \leq d-1$. Then $P$ is generated by a set $\{\tilde{a_i} \}_{i=0}^{d-2} \cup Y$, where $\pi_{d-1}(\tilde{a_i}) = a_i \in G(d-1)$, and $Y$ is a generating set for $V^{(i)}_{d-1}$
\end{proposition}

\begin{proof}
This follows from the fact that $P$ is an extension of $P(d-1)$ and $P_{d-1}$, and from \textit{a fortiori} results on presentations for group extensions (see, for instance,~\cite[Proposition 10.2.1]{Johnson-Presentations-1997}). 
\end{proof}

\begin{cor}\label{c:form-of-generating-set}
Let $P$ be an essential pattern group with $P(d-1) = G(d-1)$. If $[G(d):P] = 2^i$ , then there exists a generating set $\{ a_0w_0, a_1w_1, \ldots, a_{d-2}w_{d-2}, v^* \}$ such that $w_i \in V_{d-1}, (a_iw_i)^2 \in V^{(i)}_{d-1},$ and $\hite(v^*) = i$. 
\end{cor}

\begin{proof}
From the fact that $P(d-1) = G(d-1)$ and $[G(d):P] = 2^i$, we see that $P_{d-1} = V^{(i)}_{d-1}$. Applying Proposition~\ref{p:cheap-generating-sets}, we take a generating set $\{\tilde{a_i} \}_{i=0}^{d-2} \cup Y$ such that $Y$ generates $V^{(i)}_{d-1}$ and $\pi_{d-1}(\tilde{a_i}) = a_i \in G(d-2)$. Since $\ker \pi_{d-1} = P_{d-1} = V^{(i)}_{d-1}$, each $\tilde{a_i}$ must have the form $a_iw_i$, where $w_i \in P_{d-1}$. Since the action of $P(d-1)$ on $V_{d-1}$ is uniserial and $\hite(v^*) = i$, it follows that $Y$ is a subset of the normal closure of $\langle v^*\rangle$ with respect to $\langle \{ \tilde{a_i} \}_{i=0}^{d-2} \rangle$, so $\langle\{\tilde{a_i} \}_{i=0}^{d-2} \cup Y \rangle = \langle \{ a_0w_0, a_1w_1, \ldots, a_{d-2}w_{d-2}, v^* \} \rangle$ which completes the proof.
\end{proof}

It is worth noting that stronger conditions could be imposed on the $w_i$ in the previous proof, but we will not explore that here.

\section{Main Results}

We are now prepared to count and characterize the finitely constrained groups of nearly maximal Hausdorff dimension. 

\subsection{Finite Patterns for Finitely Constrained Groups of Nearly Maximal Hausdorff Dimension}

\begin{theorem}
If $P$ is an essential pattern subgroup of $G(d)$ such that $\Hdim(G_P) = 1 - \frac{2}{2^{d-1}}$, then $P(d-1) = G(d-1)$.
\end{theorem}

\begin{proof}
Assume that there is some $P$ such that $G_P$ has Hausdorff dimension $1-  \frac{2}{2^{d-1}}$, but $P(d-1) \neq G(d-1)$. We may assume that $d$ is the smallest pattern size such that there is such a $P$.  Observe first that we must have $\log_2 |P_{d-1}| = 2^{d-1}-2$. Let $Q = P(d-1)$, and consider $Q_{d-2}$. 
Since $P$ is an essential pattern group, so is the group $Q$, by Proposition~\ref{p:bigger-pattern-sizes-define-the-same-group}. By assumption, $Q_{d-2} \neq G(d-1)_{d-2}$, and so we have  
\[
\log_2 |Q_{d-2}| \leq 2^{d-2}-1. 
\]  However, if $\log_2 |Q_{d-2}| < 2^{d-2}-1$, then from the fact that $G_Q$ is regular branch over $Q_{d-2}$ and that $|G_Q(d)_{d-1}| = |Q_{d-2}|^2$, we would have $\log_2 |(G_Q(d))_{d-1}| < 2^{d-1}-2$. Since $P \leq G_Q(d)$, this means that we must have $Q_{d-2} = 2^{d-2}-1$. 

Then we have
\begin{align*}
|\left(G_{Q(d)}\right)_{(d-1)}| &= |Q_{d-2}|^2 \\
&= 2^{2^{d-1}-2} \\
&= |P_{d-1}|
\end{align*}

Thus, since $P_{d-1} \leq \left(G_Q(d)\right)_{d-1}$ and $P_{d-1} \geq G_Q(d)_{d-1}$, these two finite groups are actually equal. 

Applying Equation~\ref{e:finite-quotient-growth}, it follows that for all $n \geq d$, we have $|G_P(n)| = |G_Q(n)|$, and hence $G_P(n) = G_Q(n)$ since $G_P(n) \subseteq G_Q(n)$ and they are finite groups of the same size. Hence $G_{P} = G_{Q}$ by part (iii.) of Proposition~\ref{p:bigger-pattern-sizes-define-the-same-group}. This implies that $G_P$ is actually defined by patterns of size $(d-1)$, contradicting our assumption that $G_P$ was defined by patterns of size $d$. Thus, it must be the case that $P(d-1) = G(d-1)$. 
\end{proof}

\begin{cor}\label{c:essential-patern-index-four}
If $G_P$ is a finitely constrained group defined by an essential pattern subgroup $P$ of pattern size $d$ such that Hausdorff dimension equal to $1 - \frac{2}{2^{d-1}}$, then $[G(d):P] = 4$ and $P_{d-1} = V^{(2)}_{d-1}$. 
\end{cor}

\begin{proof}
Since in this case $P(d-1) = G(d-1)$, the action of $P$ on $V_{d-1}$ is uniserial, so $P_{d-1} = V^{2}_{d-1}$ by Theorem~\ref{t:practical-uniserial}.
\end{proof} 

\begin{lemma}\label{l:no-a0}
Let $P$ be an essential pattern subgroup of $G(d)$ such that $[G(d):P] = 4$. If $M$ is a maximal subgroup of $G(d)$ such that $P \leq M$, then  $M$ contains the element $a_0$, and $M$ does not contain either $a_0a_{d-1}$ or $a_{d-1}a_0$.
\end{lemma}

\begin{proof}
Notice that $P$ has index 2 in $M$, so there is a homomorphism $M \rightarrow C_2$ such that $P = \ker \phi$.  Since $\phi$ is a map onto an elementary abelian $p$-group, it follows that $\Phi(M) \subseteq P$. From Theorem~\ref{t:maximal-subgroup-structure}, $M =\ker \alpha_{J}$ for some $J \subseteq \{0, \ldots, d-1 \}$ with $d-1 \in J$.  If $a_0a_{d-1} \in P$, then it would follow that $(a_0a_{d-1})^2 = [a_0,a_{d-1}] \in P$ since the Frattini subgroup contains all squares of elements in $M$.  We would then have that $P_{d-1}$ contains the normal closure of the subgroup $< [a_0, a_{d-1}] >$, which has index 2 in $G(d)_{(d-1)}$ -- this is a contradiction. An identical argument rules out $a_{d-1}a_0$ being in $M$. Following Remark~\ref{r:aj-or-ajad-1}, $a_0 \in M$ if and only if $a_0a_{d-1} \not\in M$, so we conclude $a_0 \in M$. 
\end{proof} 

\begin{remark}\label{r:coset-representatives}

Let $P$ be an essential pattern group of pattern size $d$ such that $G_P$ has Hausdorff dimension $1 - \frac{2}{2^{d-1}}$. We know from that there are four cosets for $P_{d-1}$ in $G(d)_{d-1}$ by Corollary~\ref{c:essential-patern-index-four}. From Proposition~\ref{p:large-stabilizer-subgroups}, we can take the following standard representatives for each coset: 
\begin{enumerate}
\item  $z_0$ = the identity, with all labels trivial, representing $V_{d-1}^{(2)}$
\item  $z_1 = a_{d-1}$, representing the coset of $V_{d-1}^{(2)}$ with $\alpha_{0X^{d-2}}(v) = \sigma$ and $\alpha_{1X^{d-2}}(v) = \id$, 
\item $z_2 = a_{d-1}^{a_0}$ representing the coset of $V_{d-1}^{(2)}$ with $\alpha_{0X^{d-2}}(v) = \id$ and $\alpha_{1X^{d-2}}(v) = \sigma$, 
\item $z_3 = [a_0, a_{d-1}]$, representing the coset of $V_{d-1}^{(2)}$ with $\alpha_{0X^{d-2}}(v) = \alpha_{1X^{d-2}}(v) = \id$
\end{enumerate}

\end{remark}

\begin{proposition}\label{p:describe-generating-sets}
Let $P$ be an essential pattern group with $P(d-1) = G(d-1)$ and $[G(d):P] = 4$. Then $P$ is generated by the set \[
\{a_iz_{k_i} \}_{i=0}^{d-2} \bigcup \{ [a_1, a_{d-1}] \}
\]

where $k_i \in \{0,1,2,3\}$.

\end{proposition}

\begin{proof}
This is a direct consequence of Corollary~\ref{c:form-of-generating-set} and Remark~\ref{r:coset-representatives}.
\end{proof}

\begin{proposition}~\label{p:upper-bound}
Let $d \geq 1$. There are at most $2^{2d-3}$ essential pattern subgroups of index 4 in $G(d)$.
\end{proposition}

\begin{proof}
Let $P$ be an essential pattern group pf pattern size $d$ with $[G(d):P] = 4$. We will count the possible generating sets described in Proposition~\ref{p:describe-generating-sets}. Since distinct groups obviously can not be assigned the same generating set, this will provide an upper bound. 

Note that $a_0a_{d-1} \not\in P$ and $a_0a_{d-1}^{a_0} = a_{d-1}a_0 \not\in P$ by Lemma~\ref{l:no-a0}. So there are at most two choices of pattern on the last level corresponding to $a_0$: we either have $a_0z_0$ or $a_0z_3 \in P$ and at most four choices of coset representative for each $a_i$, $1 \leq i \leq d-2$. Thus there are at most $(2)4^{d-2} = 2^{2d - 3}$ such groups. 
\end{proof}

Our goal is now to prove that this upper bound is also a lower bound. To do so, we construct homomorphisms for which these groups are the kernels, which allows us to describe the patterns of the index 4 essential pattern subgroups. 

We observe that the action of $G(d)$ on $X^{(d)}$ extends to a left action by bijections on subsets of $X^{(d)}$. The fixed points of this action are precisely the sets of the form $X^J$ for some $J \subseteq \{0,1,\ldots,d-1\}$. We let $\Delta$ denote the symmetric difference operation on two subsets of $X^{(d)}$. 

\begin{definition}
Given a set $J$ which contains $d-1$, a \textit{decomposition subordinate to $X^J$} is a pair of sets $S, T \subset X^{(d)}$ satisfying the following properties: \begin{enumerate}
\item $S \Delta T = X^J $ \\ 
\item for any $g \in G(d)$ , we have either $g(S) = S$ or $g(S) = T$.
\end{enumerate}
\end{definition}

Note that the second condition in the definition says that $P_J$ acts as $C_2$ by permutations on the set $\{S,T\}$, which forces $S$ and $T$ to have the same cardinality. 

\begin{definition}
If $J \subseteq \{0, 1, \ldots, d-1\}$ and $S,T$ form an invariant decomposition subordinate to $X^J$, we define the set
\[
P_{S,T} = \{ g \in G(d) \mid \alpha_{S}(g) = \alpha_{T}(g) = 0 \}. 
\] 
\end{definition}

\begin{lemma}~\label{l:index-4-homs}
Let $J \subseteq \{0,1, \ldots, d-1\}$. If $S$ and $T$ form an invariant decomposition subordinate to $X^J$, then the set $P_{S,T}$ is a subgroup of $G(d)$ such that $[G(d):P_{S,T}] = 4$ and $P_{S,T}(d-1) = G(d-1)$. 
\end{lemma}

\begin{proof}
First, we show that $P_{S,T} \subseteq P_J$. Note that if $p \in P_{S,T}$, then
\begin{align*} 0 &= \alpha_{S}(p) + \alpha_{T}(p) \\ 
&= \alpha_{S \cap T^c}(p) + \alpha_{S \cap T}(p) + \alpha_{S \cap T}(p) + \alpha_{T \cap S^c}(p) \\ 
&= \alpha_{S \Delta T}(p) \\  
&= \alpha_{J}(p)
 \end{align*} 
Thus $P_{S,T} \subseteq P_J$. 

It also follows from the previous calculation that $\alpha_{S}(g) = \alpha_{T}(g)$ for all $g \in P_J$. From this and the fact that $h(S) \in \{S, T\}$ for all $h \in G(d)$, it is not hard to see that that $\alpha_{S}$ restricts to a surjective homomorphism $P_J \rightarrow C_2$ such that $P_{S,T} = \ker \alpha_{S}$. Since $P_J$ has index 2 in $G(d)$ and $P_{S,T}$ has index 2 in $P_J$, we have that $P_{S,T}$ has index 4 in $G(d)$. 

\end{proof}

\begin{proposition}~\label{p:lower-bound}
Let $d \geq 2$. There are at least $2^{2d-3}$ essential pattern subgroups of index 4 in $G(d)$.
\end{proposition}

\begin{proof}
Let $P_J$ be a maximal subgroup of $G(d)$, where  $J$ is a subset of $\{ 1, \ldots, d-1 \}$ such that $(d-1) \in J$. We count the subgroups corresponding to decompositions subordinate to $X^J$. In constructing $S$ for such a decomposition, note that we have a choice for each $j \in J$, whether to put $0X^{j-1}$ or $1X^{j-1}$ in $S$, and for each $k$ in the complement of $J$, there is a choice of whether or not to include $X^k$ in $S$. So, in total, there are $2^{d-1-|J|}$ such choices in the construction of $S$. Once $J$ and $S$ are fixed, the $T$ is determined by the fact that $T = S \Delta X^J$, and thus these choices determine the subgroup $P_{S,T}$ uniquely. Thus, for each subset $J$ of $\{1, \ldots, d-2, d-1 \}$ such that $d-1 \in J$, there are at least  $2^{|J|}2^{d-1-|J|} = 2^{d-1}$ subgroups of $P_J$ that have the desired form. Moreover, it is clear that each of these subgroups is distinct, as their sets of patterns are different. Since there are exactly $2^{d-2}$ choices of $J$ to define $P_J$, there are at least $2^{d-2}2^{d-1} = 2^{2d-3}$ such subgroups.
\end{proof}

Taken together, Proposition ~\ref{p:upper-bound} and Proposition ~\ref{p:lower-bound} yield the following theorem.

\begin{theorem}\label{t:main-theorem-1}
There are exactly $2^{2d-3}$ finitely constrained subgroups of $\Aut(X^*)$ that are defined by patterns of size $d$ and have Hausdorff dimension $1 -\frac{2}{2^{d-1}}$.
\end{theorem}

We also note that the description of patterns given above shows that these groups have additive portraits. 

\begin{theorem}\label{t:main-theorem-2}
Let $d \geq 2$. If $G_P$ is a finitely constrained subgroup of $\Aut(X^*)$ defined by patterns of size $d$ such that $\Hdim(G_P) = 1 - \frac{2}{2^{d-1}}$, then $G_P$ has additive portraits. 
\end{theorem}

\begin{proof}
Let $G_P$ be a finitely constrained group with nearly maximal Hausdorff dimension, defined by an essential pattern group $P$ with pattern size $d$. We know from Proposition~\ref{p:upper-bound} and Proposition~\ref{p:lower-bound} that there exist a subset $J \subseteq \{0,1,\ldots,d-1\}$ and sets $S,T$ subordinate to $X^J$ such that $p \in P$ if and only if $\alpha_J(p) = \alpha_{S}(p) = \id$. Hence, $g \in G_P$ if and only if $\alpha_{wX^J}(g) = \alpha_{wS}(g) = \id$. If $g,h \in G_P$ such that $\rho(g)$ and $\rho(h)$ meet this condition, then $\rho(g)\oplus\rho(h)$ clearly meets this condition as well. 
\end{proof}

\begin{remark}
Again, we note  that the recent independent work of Samoilovych~\cite{Samoilovych-Profinite-2017} provides an independent description of the portraits of some of the groups that we have just considered. In general, finitely constrained groups of nearly maximal Hausdorff dimension contain certain instances of topological closures of iterated monodromy groups, but the two classes certainly do not coincide. 
\end{remark}

\subsection{Topological Finite Generation in Finitely constrained Groups of Nearly Maximal Hausdorff Dimension}

While we are not able to completely determine the question of topological finite generation for all of the $2^{2d-3}$ groups described in the previous section, we are able to provide an upper bound for the number of topologically finitely generated groups, finitely constrained groups of Hausdorff dimension $1 - \frac{2}{2^{d-1}}$ defined by patterns of size $d$.

\subsubsection{Useful criteria for determining topological finite generation}

Bondarenko and Samoilovych~\cite{Bondarenko-Finite-2013} give the following criterion to show that a group is \textit{not} topologically finitely generated. 

\begin{proposition}[Proposition 4, ~\cite{Bondarenko-Finite-2013}]~\label{p:not-topo-fg}
Let $X$ be a finite set and let $G_P$ be a finitely constrained subgroup of $\Aut(X^*)$ defined by an essential pattern subgroup $P$ of pattern size $d$. If there exists an $n \geq d$ such that $[G_P(n), G_P(n)]$ does not contain $\Triv_{G_P(n)}(n-1)$, then $G_P$ is not topologically finitely generated. 
\end{proposition}

We note the following corollary of Proposition~\ref{p:not-topo-fg}, which also follows from a result by Siegenthaler~\cite[Theorem 2.2.9]{Siegenthaler-Discrete-2009}.

\begin{cor}\label{c:easy-cor}
Let $G_P$ be a finitely constrained subgroup of $G$ defined by an essential pattern subgroup $P$ of pattern size $d$. If there exists an $n \geq d$ and a homomorphism $\phi: G_P(n) \rightarrow C_2$ such that $G_P(n)_{n-1}$ is not contained in the kernel of $\phi$, then $G_P$ is not topologically finitely generated. 
\end{cor}

\begin{proof}
If there exists such an $n$ and such a $\phi$, then $\ker \phi$ is a maximal subgroup of $G_P(n)$ which does not contain $\Triv_{G_P(n)}(n-1)$. It follows that the Frattini subgroup $\Phi(G_P(n))$ does not contain $G_P(n)_{n-1}$, and thus $[G_P(n), G_P(n)]$ does not contain $G_P(n)_{n-1}$. Applying Proposition~\ref{p:not-topo-fg}, it follows that $G_P$ is not topologically finitely generated. 
\end{proof}

\begin{remark}\label{r:how-to-see}
A homomorphism  $\phi$ as described in Corollary~\ref{c:easy-cor} can be recognized by the fact that there are two elements of $G_P(n)_{n-1}$ for which $\phi$ takes different values. 
\end{remark}

\begin{cor}\label{c:easy-cor-2}
Let $P$ be an essential pattern group contained in $G(d)$. If there is a subgroup $K \leq P$ such that $K \cap P_{d-1}$ is trivial and $KP_{d-1}= P$, then the finitely constrained group $G_P$ is not topologically finitely generated. 
\end{cor}

\begin{proof}
Assume there is a subgroup $K \leq P$ such that $K \cap P_{d-1}$ is trivial and $KP_{d-1} = P$. Let $M$ be a maximal subgroup of $P$ such that $K \leq M$. Note that it is not possible for $M$ to also contain $P_{d-1}$, since then we would have that $M$ contains $K\Triv_P(d-1) = P$. Then $[P:M] = 2$, and the kernel of the homomorphism $\phi: P \rightarrow P/M \cong C_2$ does not contain $P_{d-1}$. Applying Corollary~\ref{c:easy-cor}, we conclude that $G_P$ is not topologically finitely generated. 
\end{proof}

\begin{remark}\label{r:split-extension}
The condition in the previous corollary is equivalent to saying that $P$ is a split extension of $P_{d-1}$ by $K$, or that $P$ is isomorphic to the semi-direct product of $K$ and $P_{d-1}$. 
\end{remark}

\begin{cor}\label{c:easy-cor-3}
Let $P$ be a subgroup of $G(d)$ with $P(d-1) = G(d-1)$. If $P$ has a maximal subgroup $Q$ which has the property that $Q(d-1) = G(d-1)$, then $G_P$ is not topologically finitely generated.
\end{cor}

\begin{proof}
Since $Q$ is maximal, we have that $P/Q \cong C_2$, and since $Q(d-1) = P(d-1) = G(d-1)$, it must be the case that $Q_{d-1}$ is a proper subgroup of $P_{d-1}$. Thus the homomorphism from $G_P(d) = P$ onto $C_2$ which has $Q$ as kernel is not constant on cosets of $G_P(d)_{d-1}$. Applying Proposition~\ref{p:not-topo-fg}, we see that $G_P$ is not topologically finitely generated.
\end{proof}

\subsubsection{Non-topologically finitely generated examples}

We now use the tools we have just developed to show that certain finitely constrained groups of nearly maximal Hausdorff dimension are not topologically finitely generated. 

\begin{proposition}\label{p:examples-not-topo-fg}
Suppose $d \geq 4$ and let $J \subseteq \{2, 3, \ldots, d-1\}$ such that $d-1 \in J$. If there exists $K \subseteq \{0, 1, \ldots, d-3 \}$ with $d-3 \in K$ such that $S = 00X^{K}$, then the maximal subgroup $P_J$ contains an essential pattern subgroup $Q$ such that $[G(d):Q] = 4$ and $G_Q$ is not topologically finitely generated. 
\end{proposition}

\begin{proof}

We let $T = X^J \Delta S$ and $Q = P_{S,T}$,

We take
\begin{align*}
S_0 &= S \cap 00X^{(d-2)} \\
S_1 &= S \cap 01X^{(d-2)} \\
T_0 &= T \cap 10X^{(d-2)} \\
T_1 &= T \cap 11X^{(d-2)} \\
\end{align*}

Now define a homomorphism $\phi: P \rightarrow C_2$ by

\[
\phi(g) = \alpha_{S_0}(g) + \alpha_{T_0}(g)
\]

We claim that $\phi$ is a homomorphism which is not constant on the cosets of $P_{d-1}$. To see that $\phi$ is a homomorphism, notice that
\[ 
\alpha_{S_0}(g) + \alpha_{S_1}(g) = \alpha_{S}(g) = 0  \text{ and } \alpha_{T_0}(g) + \alpha_{T_1}(g) = \alpha_{T}(g) = 0,
\]
it follows that the value of $\phi$ is constant under any  permutations of the collection of sets $\{ S_0, S_1, T_0, T_1 \}$.

To see that $\phi$ is not constant on cosets of $P_{d-1}$, consider the element $g$ in $P_{d-1}$ with exactly one nontrivial label on $S_0 \cap X^{d-1}$ and exactly one nontrivial label on $S_1 \cap X^{d-1}$,  and let $h$ be an element of $P_{d-1}$ with exactly two nontrivial labels in $S_0$ and all other labels trivial. We see that $\phi(g) \neq \phi(h)$, but $\pi_{d-1}(g) = \pi_{d-1}(h)$. Thus $\phi$ is a homomorphism from $P$ to $C_2$ which is not constant on cosets of $P(d-1)$. It follows from the discussion in Remark~\ref{r:how-to-see} that $G_P$ is not topologically finitely generated.
\end{proof}

\begin{cor}
For each $d \geq 4$, there are at least $2^{d-3}$ finitely constrained groups with pattern size $d$ and Hausdorff dimension $1 - \frac{2}{2^{d-1}}$ that are not topologically finitely generated.
\end{cor}

\begin{proof}
Given $d \geq 4$ and a subset $J \subseteq \{2,3,\ldots, d-1\}$ that contains $d-1$, we can produce a set $S$ that meets the hypothesis of Proposition~\ref{p:examples-not-topo-fg}. 
\end{proof}

\subsubsection{Topologically Finitely Generated Examples}

We have just demonstrated that not all finitely constrained groups defined by patterns of size $d$ with nearly maximal Hausdorff dimension are topologically finitely generated. However, as discussed in the Introduction, some known examples are. Our aim in this subsection is to construct a new family of examples of with these properties. 

First, we provide some brief additional background on self-similar groups. 

\begin{lemma}\label{l:cojugating-a-branch}
Let $r_0 = \sigma(\id,\id) \in G$. If $h$ is any element of $G$ and $x \in X$, then $\delta_{0}\left((h)^{r_0}\right) = \delta_{1}(h)$
\end{lemma}

\begin{proof}
Since $\delta_0(h) = (h,\id)$ in wreath recursion form, we see that 
\[ (\delta_0(h))^{r_0} = \sigma(\id,\id)(h,\id)\sigma(\id,\id) = (h,\id)^{\sigma} = (\id, h) = \delta_1(h) \]. 
\end{proof}

\begin{definition}
A group $K$ of tree automorphisms is \textit{self-replicating} if for each $g \in K$, there exist $h, k \in K_1$ such that $h_0 = k_1 = g$. 
\end{definition}

\begin{lemma}\label{l:sufficient-to-branch}
Let $K$ be a level-transitive, self-replicating group such that $[K:K']$ is finite. Let $T$ be a generating set for $H$. If $K$ contains $\delta_{x}([t_i,t_j])$ for all $t_i, t_j \in T$ and $x \in X$, then $K$ is a regular branch group over $K'$.
\end{lemma}

\begin{proof}
It is a standard group-theoretic fact that in this instance, the group $K'$ is the normal closure in $K$ of $\langle \{ [t_i,t_j] \} \rangle$ for $t_i, t_j \in T$. Assume $K$ contains $\delta_{0}([t_i,t_j])$ for all $t_i, t_j \in T$. For any $k \in K$, we can find $g$ such that $g = (k,g_k)$ for some $g_k \in \Aut(X^*)$, and $g^{-1} = (k^{-1},g_k^{-1})$, so 
\[
g\delta_{0}([t_i,t_j])g^{-1} = (k,g_k)([t_i,t_j],\id)(k^{-1},g_k^{-1}) = (k[t_i,t_j]k^{-1},\id) = \delta_0([t_i,t_j]^k,\id)
\]

We have obtained all elements of a generating set for $K'$ in the image of the map $\delta_0$. Hence, taking products, we can obtain $\delta_0(h)$ for any $h \in K'$. We repeat the argument to obtain $\delta_1(h)$ for each $h \in K'$. As we assumed that $[K:K']$ is finite and we know that $K'$ is a normal subgroup of $K$, this shows that $K$ is a regular branch group over $K'$. 
\end{proof}

For the remainder of this section, fix $k \geq 1$ and define the set $S \subseteq \Aut(X^*)$ as
\[
S = \{ r_0 = \sigma(\id,\id), r_i = (r_{i-1}, \id), 1 \leq i \leq k, b_0 = (r_{k}, b_1), b_1 = (r_{k}, b_2), b_2 = (\id, b_0)  \}
\]

Also, we let $H$  be the group generated by $S$. The reader might notice that this generating set $S$ is modeled after the generating set of the first Grigorchuk group in Example~\ref{e:grigorchuk-automaton}.

Now let us record some basic properties of $H$.

\begin{lemma}\label{l:group-is-self-replicating}
The group $H$ is self-replicating. 
\end{lemma}

\begin{proof}
We see that $(a_i)_0 = a_{i+1}$ for $0 \leq i \leq k-1$. Also, we have $(b_1)_0 = a_{k}$, and the section of $(b_t)^{a_0}$ at $0$ is given by $b_{t+1}$ (with indices interpreted modulo 3). 
\end{proof}

\begin{lemma}\label{l:B-group-is-abelian}
The group $B = \langle b_0, b_1, b_2 \rangle$ is an elementary abelian 2-group, isomorphic to $C_2 \times C_2$.
\end{lemma}

\begin{proof}
First, we will show that each generator of $B$ have order two. We calculate that for $j = 0,1,2$, $b_j^2 = (\id, b_{j+1}^2)$, (where the indices $j+1$ are interpreted modulo 3), so $b_j^2$ has nontrivial labels on words of the form $1w$ where $b_{j+1}^2$ has nontrivial labels. Then, noting that $(b_j^2)_{(\epsilon)} = (b_j^2)_{(0)} = (b_j^2)_{(1)} = \id$ leads to a straightforward induction argument that $b_j^2 = \id$ for $j = 0,1,2$. Similarly, we calculate that for $0 \leq i,j, \leq 2$, we have
\[
[b_i,b_j] = (1,[b_{i+1},b_{j+1]}])
\]
so
\[
[b_i,b_j]_{(\epsilon)} = [b_i,b_j]_{(0)} = [b_i,b_j]_{(1)} = \id.
\]
Again, it is now routine to use induction to establish that $[b_i,b_j] = \id$ for all $0 \leq i,j \leq 2$, which completes the proof. 
\end{proof}

\begin{proposition}\label{p:group-is-branch}
The group $H$ is a regular branch group, branching over its commutator subgroup $H'$.
\end{proposition}

\begin{proof}
In view of Lemma~\ref{l:sufficient-to-branch} it suffices to prove $\delta_{x}([g,h]) \in H'$ for $g,h \in S$. We do not need to consider commutators of elements of $B$, since this group is abelian by Lemma~\ref{l:B-group-is-abelian}.

Let us also note that if $\delta_0(h) \in H'$, then we can obtain $\delta_{1}(h)$ by conjugating $\delta_0(h)^{(r_0)}$, and vice versa, so it will suffice to obtain either $\delta_0(h)$ or $\delta_1(h)$ for each $h \in S$. 

For $0 \leq i,j \leq k-1$, we have 
\[
[r_{i+1},r_{j+1}] = ([r_i,r_j],\id) = \delta_{0}([r_i,r_j]).
\]

We note also that 

\[
[r_{i+1},  b_0] = ([r_i,r_{k}], \id) = \delta_{0}([r_i,r_k])
\]

and

\[
[r_{i+1}^{r_0},b_t] = (\id,[r_i,b_{t+1}]) = \delta_{1}[r_i,b_{t+1}])
\]

where the index $t+1$ is taken modulo 3. Using conjugation by $r_0$, it clearly follows that $\delta_0([r_i, b_t]) \in H'$ for all relevant values of $0 \leq i \leq k-1$ and $0 \leq t \leq 2$.

It remains to show that $\delta_0([a_k,b_t]) \in H'$ for $0 \leq t \leq 2$. We note that 

\[
[r_k,b_0] = ([r_{k-1},r_k],\id) = [r_k,b_1]
\]

and

\[
[r_k, b_2] = (\id,\id) = \id
\]

so it remains only to find $\delta_0([r_k,b_1]) \in H'$. Another calculation using the fact that $b_0$ and $b_1$ have order two reveals that

\[
[b_1,b_0^{r_0}] = (b_1b_0^{r_0})^2 = ([r_k,b_1],[b_2,r_k]) = ([r_k,b_1], \id)
\]

Thus, by Lemma~\ref{l:sufficient-to-branch}, $H$ is a regular branch group, branching over its commutator subgroup $H'$.

\end{proof}

\begin{lemma}\label{l:support-for-elements}
We have
\begin{align*}
\supp(r_i) &= 0^i, 0 \leq i \leq k \\
\supp(b_0) &= \{ 1^n0^{k+1} \mid n \equiv 0,1 \pmod{3} \} \\
\supp(b_1) &= \{ 1^n0^{k+1} \mid n \equiv 0,2 \pmod{3} \} \\
\supp(b_2) &= \{ 1^n0^{k+1} \mid n \equiv1,2 \pmod{3} \} \\
\end{align*}
\end{lemma}

\begin{proof}
The fact that $\supp(r_i) = 0^i$ is an easy exercise in induction, following from the observations that $\supp(r_0) = \epsilon = 0^0$ and $r_{j+1} = \delta_0(r_j)$ for $0 \leq j \leq k-1$. We see immediately that $0^{k+1} \in \supp(b_0)$ since $0^k \in \supp(a_{k})$ and $b_0 = (r_k, b_1)$, and similar reasoning gives that $0^{k+1} \in \supp(b_1)$, from which it follows that that $10^{k+1} \in \supp(b_0)$. Using this as a base case, it immediately follows recursively that $\supp(b_t)$ has the desired form for $t \in \{0,1,2\}$.
\end{proof}

\begin{lemma}\label{l:finite-quotient-patterns}
Let $d = k+4$. 
We have 
\begin{align*}
\pi_{d-1}(r_i) = a_i, 0 \leq i \leq k \\
\pi_{d-1}(b_1) = a_{d-3} \\
\pi_{d-1}(b_2^{r_0}) = a_{d-2} \\
\end{align*}
\end{lemma}

\begin{proof}
Using Lemma~\ref{l:support-for-elements}, we see that $\pi_{d-1}(r_i) = a_i$ for $0 \leq i \leq k$. Since $\supp(\pi_{d-1}(b_1)) = X^{(d-1)} \cap \supp(b_1) = \{ 0^{k+1} \}$, it follows that $\pi_{d-1}(b_1) = a_{k+1} = a_{d-3}$. Similarly, we have
\[
\supp(\pi_{d-1}(b_2)) = \supp(b_2) \cap X^{(d-1)} = 10^{k+1}
\]
so $\supp(\pi_{d-1}(b_2^{r_0})) = 0^{k+2} = a_{d-2}$. This completes the proof.  
\end{proof}

\begin{proposition}\label{p:important-facts}
Let $d = k + 4$. Then we have the following.
\begin{itemize}
\item[(i.)] $H(d-1) = G(d-1)$
\item[(ii.)] $H/H'$ is isomorphic to $(C_2)^{d-1}$
\item[(iii.)] $H_{d-1} \subseteq H'$. 
\item[(iv.)] If $J = \{d-3, d-2, d-1 \}$ and $P_J = \ker \alpha_J$, then $H(d) \subseteq P_J$.
\item[(v.)] $H(d)_{d-1} \subseteq V^{(2)}_{d-1}$
\item[(vi.)] $[P_J:H(d)] = 2$ and $[G(d):H(d)] = 4$
\end{itemize}
\end{proposition}

\begin{proof}
(Proof of i.) By Lemma~\ref{l:finite-quotient-patterns}, we see that $H(d-1)$ contains $\{a_i\}_{i=0}^{d-2}$, and this is a generating set for $G(d-1)$. (Proof of ii.) Since $H$ is generated by the $d-1$ elements $\{a_0,\ldots,a_{k},b_0,b_1\}$, all of which have order two, its largest abelian quotient is at most $(C_2)^{d-1}$ by the Burnside Basis Theorem(see~\cite[Proposition 13.2.1]{Johnson-Presentations-1997}) . On the other hand, $\lambda: h \rightarrow [\alpha_k(h)]_{k=0}^{d-1}$ is a surjective homomorphism onto $(C_2)^{d-1}$, so its abelianization must contain $(C_2)^{d-1}$ as well. (Proof of iii.) This is a consequence of the fact that $\lambda$ is equal to the abelianization map of $G(d-1)$ composed with $\pi_{d-1}$.  (Proof of iv.) It follows immediately from Lemma~\ref{l:support-for-elements} that $\alpha_{d-3}(s) + \alpha_{d-2}(s) + \alpha_{d-1}(s) = 0$ for all $s \in S$. Hence $\alpha_{J}(h) = 0$ for all $h \in H$, from which it follows that $H(d) \subseteq P_J$ in $G(d)$. (Proof of v.) Taking the commutator $[r_1^{r_0},b_1]$, we obtain 
\[
r_1^{r_0}b_1r_1^{r_0}b_1 = (\id,[r_0,b_2])
\]
So, using the wreath decomposition of elements in $H$ and $G(d)$, we have
\[
\pi_{d}([r_1^{r_0},b_1]) = (\id, \pi_{d-1}([r_0,b_2]) = (\id,[a_0,a_{d-2}]) = [a_1,a_{d-1}].
\]
Thus $[a_1,a_{d-1}] \in H_{d-1}$. Since $H(d-1) = G(d-1)$, the action of $H$ on $H_{d-1}$ is uniserial, and we know that $H(d)_{d-1} \subseteq V^{(2)}_{d-1}$ by Corollary~\ref{c:normal-closure-height} and Corollary~\ref{c:large-height-elements}. (Proof of vi.) From (v.), we know that $[P_J:H(d)] \geq 2$. Now observe that the map $\beta: P_J \rightarrow C_2$ given by $\alpha_{0X^{d-4}}(p) + \alpha_{0X^{d-3}}(p) + \alpha_{1X^{d-2}}(p)$ is a well-defined homomorphism, following the same arguments as given in Lemma~\ref{l:index-4-homs}. It is not hard to see that $P_J \subseteq \ker \beta$, which shows that $[P_J:H(d)] \geq 2$. Thus $[P_J:H(d)] = 2$, and $[G(d):H(d)] = 4$ since $P_J$ is a maximal subgroup of $G(d)$ by Theorem~\ref{t:maximal-subgroup-structure}.
\end{proof}

\begin{theorem}\label{t:conclusion}
Let $k \geq 1$. Let $H$ be the subgroup of $\Aut(X^*)$ generated by the finite state automaton
\[
\{ r_0 = \sigma(\id,\id), r_{i} = (r_{i-1}, \id), \begin{tiny} (1 \leq i \leq k), \end{tiny} b_0 = (r_{k}, b_1), b_1 = (r_{k}, b_2), b_2 = (\id, b_0)  \}.
\]

Then $\overline{H}$, the topological closure of $H$ in $\Aut(X^*)$ is a topologically finitely generated, finitely constrained group defined by patterns of size $d = k+4$, and $\Hdim(\overline{H}) = 1 - \frac{2}{2^{d-1}}$ 
\end{theorem}

\begin{proof}
The group $\overline{H}$ is obviously topologically finitely generated, since by definition it is the topological closure in $\Aut(X^*)$ of the finitely generated group $H$. The fact that $\overline{H}$ is finitely constrained follows from Proposition~\ref{p:important-facts}, part (iii.), and Theorem~\ref{t:branch-equals-finitely-constrained}, while Proposition~\ref{p:important-facts}, part (vi.) implies that $\Hdim(\overline{H}) = 1 - \frac{2}{2^{d-1}}$.
\end{proof}

\section{Conclusion}

\setcounter{theorem}{0}

 It seems to us that the present state of knowledge on finitely constrained groups is just the tip of the iceberg. Accordingly, we would like to take the opportunity to pose several questions for future consideration. 

Recall from Proposition~\ref{p:all-haus-dim-occur}, that it is possible to realize every possible value of Hausdorff dimension a finitely constrained group. The examples constructed in the proof of that Theorem all have additive portraits, but none are topologically finitely generated (this follows from Corollary~\ref{c:easy-cor-2} and Remark~\ref{r:split-extension}). 

\begin{question}
For which values of $d$ and $k$ does there exist a \textit{topologically finitely generated}, finitely constrained group defined by patterns of size $d$ and having Hausdorff dimension $1 - \frac{k}{2^{d-1}}$? 
\end{question}

\begin{question}
Are there any restrictions on the Hausdorff dimension of a topologically finitely generated, finitely constrained group with \textit{additive portraits}? 
\end{question}

In this work, we showed that there are exactly $2^{2d-3}$ finitely constrained groups defined by patterns of size $d$ and having Hausdorff dimension $1 - \frac{2}{2^{d-1}}$. It may be interesting to consider this quantity for other similar functions of $d$. 

\begin{question}
For a fixed $k$, how many finitely constrained groups defined by patterns of size $d$ have Hausdorff dimension $1 - \frac{k}{2^{d-1}}$?
\end{question}

Note that this question can be recast as a question about finite groups. 

\begin{question}
For a fixed $k$, how many essential pattern groups $P$ in $G(d)$ are there with $\log_2 |P_{d-1}| = 1 - \frac{k}{2^{d-1}}$?
\end{question}

It would still be interesting, but perhaps more approachable, to restrict the previous question to special cases, such as groups with $P(d-1) = G(d-1)$, or to finitely constrained groups with additive portraits. 

Also in this work, we determined an upper bound, but not an exact count, for the number of topologically finitely generated, finitely constrained groups with Hausdorff dimension $1 - \frac{2}{2^{d-1}}$. In the spirit of this observation, we pose the following question related to the asymptotic proportion of topologically finitely generated groups in this context.

\begin{question}
Let $N_{d}(k)$ be the total number of finitely constrained groups defined by patterns of size $d$ that have Hausdorff dimension $1 - \frac{k}{2^{d-1}}$. Let $T_{d}(k)$ be the number of \textit{topologically finitely generated}, finitely constrained groups defined by patterns of size $d$ that have Hausdorff dimension $1 - \frac{k}{2^{d-1}}$. If we fix $k$,  what is the quantity
\[
\lim_{d \rightarrow \infty} \frac{T_d(k)}{N_d(k)}?
\]

Does the limit even exist? 
\end{question}

Note that from the results of~\cite{Penland-Finitely-2016}, we know that $T(d,1) = 0$ for all $d$, but this is the only case of which the author is aware where anything is known.

Finally, all of the topologically finitely generated, finitely constrained examples in the literature of which the author is aware have ``large'' Hausdorff dimension as a function of pattern size, i.e. nearly all of the examples presented in the literature of topologically finitely generated, finitely constrained groups defined by patterns of size $d$ have Hausdorff dimension greater than or equal to $1 - \frac{3}{2^{d-1}}$. 

\begin{question}
Is there a nontrivial lower bound, as a function of $d$, on the Hausdorff dimension of a topologically finitely generated, finitely constrained group, defined by patterns of size $d$?
\end{question}

At present, the best-known answer to this last question is $1 - \frac{3}{2^{d-1}}$, from examples due to \v{S}uni{\'c} in~\cite{Sunik-Hausdorff-2007}, or from separate examples in Bartholdi and Nekrashevych~\cite{Bartholdi-Iterated-2008}.  In a forthcoming work~\cite{Penland-Small-Preprint}, the author will provide some improvement on this bound, by constructing a family of topologically finitely generated, finitely constrained groups with pattern size $d \geq 5$ and Hausdorff dimension strictly less than $\frac{1}{2}$. 

\bibliography{main}{}
\bibliographystyle{plain}

\end{document}